\documentclass{amsart}
\usepackage{amssymb,amsmath,latexsym}
\usepackage{amsthm}
\usepackage{fontenc}
\usepackage{amssymb}
\numberwithin{equation}{section}

\newtheorem{theorem}{Theorem}[section]
\newtheorem{corollary}{Corollary}[theorem]
\newtheorem{proposition}[theorem]{Proposition}
\newtheorem{lemma}[theorem]{Lemma}

\begin{document}
\author{Alexander E Patkowski}
\title{On integrals associated with the Free particle wave packet}

\maketitle
\begin{abstract}We discuss some properties of integrals associated with the free particle wave packet, $\psi(x,t),$ which are solutions to the time-dependent Schr$\ddot{o}$dinger equation for a free particle in one dimension. Some noteworthy discussion is made in relation to integrals which have appeared in the literature. We also obtain formulas for half-integer arguments of the Riemann zeta function.\end{abstract}



\section{Introduction and Transformation Properties}
In quantum mechanics, the Schr$\ddot{o}$dinger equation is a well-known refinement of Newton's second law which can be paraphrased from the simple equation $F=ma$ [8]. That is, \begin{equation}i\hslash\frac{\partial \psi(x,t)}{\partial t}=-\frac{\hslash^2}{2m}\frac{\partial^2 \psi(x,t)}{\partial x^{2}},\end{equation}
where $\hslash$ is Planck's constant, and $\psi(x,t)$ is the wave function. The solution to (1.1) for the free particle wave packet, weighted by a momentum amplitude $\phi(z),$ is given by the integral
\begin{equation}\psi(x,t)=\int_{-\infty}^{+\infty}\phi(z)e^{izx-i\hslash z^2t/2m}dz.\end{equation}
Convergence of this integral is typically not an issue provided $\Im(t)<0.$ In general, (1.2) is typically difficult to evaluate for most $\phi(z)$ and is done numerically, but a few classical examples include the Gaussian wave packet and ``bouncing" wave packet [8].
\par The purpose of this paper is to present some further properties and evaluations of (1.2), and note some connections with other areas of special functions. In particular, our study focuses on certain transformation methods and asymptotic expansions of (1.2). We also add on some observations of Glasser [4]. We shall use the notation $p(x) \sim q(x)$ to imply the validity of the limit $\lim_{x\to \infty}p(x)/q(x)=1.$ We also keep the standard notation for Hermite polynomials $H_n(x)$ [1].
\par We will produce more general evaluations than the following two integrals [5, pg.806]:
\begin{equation} \int_{0}^{\infty}e^{-az^2}H_{2n}(\sqrt{a}z)\cos(\sqrt{2}\beta z)dz=\frac{(-1)^n2^{n-1}}{a^{n+1/2}}\sqrt{\pi}\beta^{2n}e^{-\beta^2/2a},\end{equation}
\begin{equation} \int_{0}^{\infty}e^{-az^2}H_{2n+1}(\sqrt{a}z)\sin(\sqrt{2}\beta z)dz=\frac{(-1)^n2^{n-1/2}}{a^{n+1}}\sqrt{\pi}\beta^{2n+1}e^{-\beta^2/2a}.\end{equation}
\begin{lemma} Define the function $g_n(a,b,x)$ by
\begin{equation}g_n(a,b,x):=(\frac{ib}{2x})^{2n}\frac{1}{4}\sqrt{\frac{\pi}{x}}e^{-(b^2+a^2)/(4x)}e^{ab/(2x)}\sum_{k\ge0}^{2n}\binom {2n}{k}(-\frac{\sqrt{x}}{b})^kH_{k}(\frac{a}{\sqrt{4x}}).\end{equation}
Then we have that 
\begin{equation}\int_{0}^{\infty}e^{-xz^2}z^{2n}\cos(az)\cos(bz)dz=g_n(a,b,x)+g_n(a,-b,x), \end{equation}
and
\begin{equation}\int_{0}^{\infty}e^{-xz^2}z^{2n}\sin(az)\sin(bz)dz=g_n(a,b,x)-g_n(a,-b,x). \end{equation}
\end{lemma}
\begin{proof}
From I.S. Gradshteyn and I.M. Ryshik [5]. we have for $\Re(x)>0,$
\begin{equation}\int_{0}^{\infty}e^{-xz^2}\sin(az)\sin(bz)dz=\frac{1}{4}\sqrt{\frac{\pi}{x}}\left(e^{-(a-b)^2/(4x)}-e^{-(a+b)^2/(4x)}\right),\end{equation}

\begin{equation}\int_{0}^{\infty}e^{-xz^2}\cos(az)\cos(bz)dz=\frac{1}{4}\sqrt{\frac{\pi}{x}}\left(e^{-(a-b)^2/(4x)}+e^{-(a+b)^2/(4x)}\right),\end{equation}
If we differentiate (1.9) $2n$ times relative to $a,$ we find that the Leibniz rule gives us
$$(-1)^n\int_{0}^{\infty}e^{-xz^2}z^{2n}\cos(az)\cos(bz)dz=\frac{1}{4}\sqrt{\frac{\pi}{x}}\left(\frac{\partial^{2n}}{\partial a^{2n}}e^{-(a-b)^2/(4x)}+\frac{\partial^{2n}}{\partial a^{2n}}e^{-(a+b)^2/(4x)}\right)$$
$$=\frac{1}{4}\sqrt{\frac{\pi}{x}}e^{-b^2/(4x)}[\sum_{k\ge0}^{2n}\binom {2n}{k}(\frac{-1}{\sqrt{4x}})^kH_{k}(\frac{a}{\sqrt{4x}})\frac{\partial^{2n-k}}{\partial a^{2n-k}}(e^{ab/(2x)})+$$

$$\sum_{k\ge0}^{2n}\binom {2n}{k}(\frac{-1}{\sqrt{4x}})^kH_{k}(\frac{a}{\sqrt{4x}})\frac{\partial^{2n-k}}{\partial a^{2n-k}}(e^{-ab/(2x)})]$$

$$=\frac{1}{4}\sqrt{\frac{\pi}{x}}e^{-(b^2+a^2)/(4x)}[e^{ab/(2x)}\sum_{k\ge0}^{2n}\binom {2n}{k}(\frac{-1}{\sqrt{4x}})^kH_{k}(\frac{a}{\sqrt{4x}})(\frac{b}{2x})^{2n-k}+$$

$$e^{-ab/(2x)}\sum_{k\ge0}^{2n}\binom {2n}{k}(\frac{-1}{\sqrt{4x}})^kH_{k}(\frac{a}{\sqrt{4x}})(\frac{-b}{2x})^{2n-k}]$$

$$=(\frac{b}{2x})^{2n}\frac{1}{4}\sqrt{\frac{\pi}{x}}e^{-(b^2+a^2)/(4x)}[e^{ab/(2x)}\sum_{k\ge0}^{2n}\binom {2n}{k}(-\sqrt{x})^kH_{k}(\frac{a}{\sqrt{4x}})b^{-k}+$$
$$e^{-ab/(2x)}\sum_{k\ge0}^{2n}\binom {2n}{k}(-\sqrt{x})^kH_{k}(\frac{a}{\sqrt{4x}})(-b)^{-k}]$$

$$=(\frac{b}{2x})^{2n}\frac{1}{4}\sqrt{\frac{\pi}{x}}e^{-(b^2+a^2)/(4x)}[e^{ab/(2x)}\sum_{k\ge0}^{2n}\binom {2n}{k}(-\frac{\sqrt{x}}{b})^kH_{k}(\frac{a}{\sqrt{4x}})+$$
$$e^{-ab/(2x)}\sum_{k\ge0}^{2n}\binom {2n}{k}(\frac{\sqrt{x}}{b})^kH_{k}(\frac{a}{\sqrt{4x}})^k].$$

The remainder of the proof follows after applying the definition of $g_n(a,b,x).$
\end{proof}

Now we note that if $\psi(x,t)$ is a solution of (1.1), then so are its derivatives. That is $\dot{\psi}_{n}(x,t)=\frac{\partial^n}{\partial x^n}\psi(x,t),$ for each $n\in\mathbb{N}$ is a solution. Since $\dot{\psi}_{n}(x,t)$ is a natural refinement to the integral (1.2), we wish to consider transformation properties of this integral for general $n\in\mathbb{N}.$
\begin{proposition} Let $\bar{\phi}_c(z)$ respectively $\bar{\phi}_s(z)$ be the Fourier cosine and sine transforms of $\phi(z),$ respectively. We have, for an even momentum amplitude $\phi(z),$
\begin{equation}\dot{\psi}_{2n}(x,t)=\int_{0}^{\infty}e^{-it\hslash z^2/2m}z^{2n}\cos(xz)\phi(z)dz=\int_{0}^{\infty}\bar{\phi}_c(z)(g_n(z,\frac{it\hslash}{2m},x)+g_n(z,\frac{-it\hslash}{2m},x))dz,\end{equation}
and for an odd momentum amplitude $\phi(z),$
\begin{equation}\dot{\psi}_{2n}(x,t)=\int_{0}^{\infty}e^{-it\hslash z^2/2m}z^{2n}\sin(xz)\phi(z)dz=\int_{0}^{\infty}\bar{\phi}_s(z)(g_n(z,\frac{it\hslash}{2m},x)-g_n(z,\frac{-it\hslash}{2m},x))dz,\end{equation}
And moreover, the integrals may be further transformed using the symmetry property $g_n(a,b,x)=g_n(b,a,x).$
\end{proposition}
\begin{proof}We do the even case and leave the odd case to the reader. So assume $\phi(z)$ is even, then by definition of $\dot{\psi}_{2n}(x,t)$
\begin{equation}\dot{\psi}_{2n}(x,t)=2\frac{\partial^{2n}}{\partial x^{2n}}\int_{0}^{+\infty}\phi(z)\cos(zx)e^{-i\hslash z^2t/2m}dz\end{equation}
\begin{equation}=2(-1)^n\int_{0}^{+\infty}\phi(z)z^{2n}\cos(zx)e^{-i\hslash z^2t/2m}dz.\end{equation}
Now applying Parseval's theorem for Fourier cosine transforms, and Lemma 1.1 we obtain the even case.
\end{proof}
A remarkable fact, is the connection with Ramanujan's work [2, pg.309, eq.(14.3.1)], [7], upon setting $n=0$ in Proposition 1.2, which we note in the following.
\begin{corollary} If $\phi(z)$ and $\bar{\phi}(z)$ are self-reciprocal Fourier transforms (say either even or odd), then the wave function has
the transformation property $\psi(x,t)=\sqrt{\frac{2m}{\hslash ti}}e^{i\pi x^2m\hslash/(2t\hslash)}\psi(x2m/\hslash t,2m/it\hslash).$
\end{corollary}
Ramanujan's result is actually from choosing $\phi(z)=1/\cosh{\pi z},$ which is known to be its own Fourier cosine transform. Glasser has given a detailed account of the non-gaussian amplitude $1/\cosh(\alpha (z-z_0))$ and the associated wave function, including its behavior [4].
\begin{lemma} ([6, pg.75]) Consider the integral
$I(x):=\int_{a}^{b}e^{ixt}f(z)dz,$
Assume $a, b,$ and $f(t)$ are independent of $x>0.$ Suppose that $f(t)$ and all its derivatives are continuous in $[a,b].$ Then,
by integration by parts, we have that
\begin{equation} I(x)=\sum_{i=0}^{n-1}(\frac{i}{x})^{i+1}\left(e^{iax}f^{(i)}(a)-e^{ibx}f^{(i)}(b)\right)+r_n(x),\end{equation}
where $$r_n(x)=\left(\frac{i}{x}\right)^n\int_{a}^{b}e^{ixz}f^{(n)}(z)dz.$$
\end{lemma}
This lemma, while simple and archaic, has interesting applications to integral evaluations in limiting cases. It is also known that $r_n(x)=o(x^{-n})$ by the Riemann-Lebesgue lemma.
\begin{theorem} Let $\phi(z)$ and all its derivatives be continuous functions on $\mathbb{R},$ and $\phi(z)\rightarrow0$ when $z\rightarrow\pm\infty.$ For natural numbers $n\in\mathbb{N},$ we have that  
$$\int_{-\infty}^{+\infty}e^{ixz-i\hslash z^2t/2m}\phi(z)dz$$
\begin{equation}=(\frac{i}{x})^n\int_{-\infty}^{+\infty}e^{ixz-i\hslash z^2t/2m}\sum_{k\ge0}^{n}\binom {n}{k}(\sqrt{it\hslash/2m})^{k}(-1)^kH_k(\sqrt{it\hslash/2m}z)\phi^{(n-k)}(z)dz. \end{equation}
\end{theorem}
 \begin{proof} Select $\phi(z)$ to be as in the theorem, and then apply Lemma 1.3 with $f(z)=e^{-i\hslash z^2t/2m}\phi(z).$ Using the Leibniz rule and letting $a\rightarrow-\infty,$ $b\rightarrow+\infty$ we complete the proof.\end{proof}
This theorem tells us that a wave function with amplitude which satisfies the hypothesis of theorem, admits an expansion of the form $\psi(x,t)=\sum_{1\le k\le n}\ddot{\psi_{k}}(x,t),$ where each wave function $\ddot{\psi_{k}}(x,t)$ is weighted by an amplitude involving $H_k(z)\phi^{(n-k)}(z).$
 \section{Asymptotic properties}
 This section is devoted to some asymptotic observations on (1.2). A simple observation may be made in the following. Suppose $\phi(z)$ is even, then
 \begin{equation}\psi(x,t)=2\int_{0}^{\infty}\phi(z)\cos(xz)e^{-i\hslash z^2t/2m}dz=\sum_{n\ge0}\frac{(it\hslash/2m)^{n}}{n!}\frac{\partial^{2n}}{\partial a^{2n}}\left(\bar{\phi}_c(z)\right),\end{equation}
 where
 \begin{equation}\bar{\phi}_c(a)=\int_{0}^{\infty}\phi(z)\cos(az)dz.\end{equation}
The term $\frac{\partial^{2n}}{\partial a^{2n}}\left(\bar{\phi}_c(a)\right)$ may be easily evaluated in most instances. However, in some cases interchanging the order of this term with the power series gives us asymptotic equivalence rather than equality. 
 \begin{theorem} Let $\Re(\beta)>0.$ As $t\rightarrow \infty,$ we have that
 \begin{equation} \int_{0}^{\infty}\frac{\cos(xz)}{\cosh(\beta z)}e^{-i\hslash z^2t/2m}dz\sim \frac{\pi}{2\beta}\sum_{n\ge0}(-1)^ne^{-(2n+1)x+\frac{\pi^2}{4\beta^2}i(2n+1)^2t\hslash/2m}.\end{equation}
 \end{theorem}
 \begin{proof} Suppose $\phi(z)=1/\cosh(\pi z)$ then, since [5]
$$\int_{0}^{\infty} \frac{\cos(wt)}{\cosh(\beta t)}=\frac{\pi}{2\beta\cosh{\frac{\pi w}{2\beta}}},$$
 the right side of (2.1) becomes 
 $$\sum_{n\ge0}\frac{(it\hslash/2m)^{n}}{n!}\frac{\partial^{2n}}{\partial a^{2n}}\left(\bar{\phi}_c(a)\right)=\sum_{n\ge0}\frac{\pi^{2n+1}(it\hslash/2m)^{n}}{(2\beta)^{2n+1}n!}\sum_{r\ge0}(-1)^r(2r+1)^{2n}e^{-(2r+1)x}$$
 $$\sim  \frac{\pi}{2\beta}\sum_{n\ge0}(-1)^ne^{-(2n+1)x+\frac{\pi^2}{4\beta^2}i(2n+1)^2t\hslash/2m}.$$
 \end{proof}
We remark that this formula may also be obtained from dividing both sides of an equation found in [7], [4, eq.(10)] by the series on the right side of (2.3), and then letting $t\rightarrow\infty.$ However, we believe our simple proof gives a more direct approach without first finding explicit formulae. It is also possible to include Glaisher's work [3] which is intimately related to Ramanujan's work on integrals. Recall the 
integral [3, pg. 336, eq.(21)]
\begin{equation} \frac{1}{2}\int_{0}^{\infty}\frac{\cosh(\frac{\pi}{2}\sqrt{\frac{z}{2}})\cos(\frac{\pi}{2}\sqrt{\frac{z}{2}})}{\cosh(\frac{\pi}{2}\sqrt{\frac{z}{2}})+\cos(\frac{\pi}{2}\sqrt{\frac{z}{2}})}\cos(xz)dz=\sum_{n\ge0}(-1)^n(2n+1)e^{-(2n+1)^2x}.\end{equation} Then if we mimic the proof of Theorem 2.1 with 
$$\phi(z)=\frac{\cosh(\frac{\pi}{2}\sqrt{\frac{z}{2}})\cos(\frac{\pi}{2}\sqrt{\frac{z}{2}})}{\cosh(\frac{\pi}{2}\sqrt{\frac{z}{2}})+\cos(\frac{\pi}{2}\sqrt{\frac{z}{2}})},$$
we obtain the following theorem.
 \begin{theorem} As $t\rightarrow \infty,$ we have that
 \begin{equation} \int_{0}^{\infty}\cos(xz)\frac{\cosh(\frac{\pi}{2}\sqrt{\frac{z}{2}})\cos(\frac{\pi}{2}\sqrt{\frac{z}{2}})}{\cosh(\frac{\pi}{2}\sqrt{\frac{z}{2}})+\cos(\frac{\pi}{2}\sqrt{\frac{z}{2}})}e^{-iz^2t\hslash/2m}dz\sim \sum_{n\ge0}(-1)^n(2n+1)e^{-(2n+1)^2x+\frac{1}{4}i(2n+1)^4t\hslash/2m}.\end{equation}
 \end{theorem}

\section{An application in Number Theory of an associated integral}
For this section we recall the Riemann zeta function $\zeta(s)$ for $\Re(s)>1$ is given by the series $\sum_{n\ge1}n^{-s}$ [1]. Glaisher [3] states the integral (page 342, but we corrected this formula because we believe the sum should be over $n\ge1,$ not $n\ge0,$ in his proof to get $(e^{x^2}+1)^{-1}$):
\begin{equation}\sum_{n\ge1}(-1)^{n-1}\frac{e^{-b^2/n}}{\sqrt{n}}=\frac{2}{\sqrt{\pi}}\int_{0}^{\infty}\frac{\cos(2bx)dx}{1+e^{x^2}}.\end{equation}
As it turns out, we may extend our study of (1.3) and Hermite polynomials to obtain a formula for half-integer arguments of the Riemann zeta function.

\begin{theorem} Let $h_{k,m}(b)$ be given in (3.5), and let $l_{k,m}(b)$ be given by
$$2^{-2m}\frac{\sqrt{\pi}}{2}\sum_{n\ge1}\frac{1}{\sqrt{n}}\frac{e^{-b^2/n}H_m(\sqrt{\frac{1}{n}}b)}{n^{m/2}}=\sum_{n\ge1}l_{n,m}(b).$$

 We have,
\begin{equation} \sum_{n\ge1}\frac{n^{2m}}{e^{n^2}+1}=\Gamma(m+\frac{1}{2})(1-2^{\frac{1}{2}-m})\zeta(m+\frac{1}{2})+2\sum_{n\ge1}\sum_{k\ge1}h_{k,m}(\pi n),\end{equation}

and
\begin{equation} \sum_{n\ge1}\frac{n^{2m}}{e^{n^2}-1}=\Gamma(m+\frac{1}{2})\zeta(m+\frac{1}{2})+2\sum_{n\ge1}\sum_{k\ge1}l_{k,m}(\pi n).\end{equation}
\end{theorem}
\begin{proof} We give the proof of (3.2) using Glaisher's integral (3.1) and leave (3.3) to the reader.
Taking $2m$-th derivative of (3.1) with respect to $b$ of this formula gives ($n\ge1$)
\begin{equation}\sum_{n\ge1}(-1)^{n-1}\frac{1}{\sqrt{n}}\frac{(-1)^me^{-b^2/n}H_m(\sqrt{\frac{1}{n}}b)}{n^{m/2}}=(-1)^m2^{2m}\frac{2}{\sqrt{\pi}}\int_{0}^{\infty}\frac{x^{2m}\cos(2bx)dx}{1+e^{x^2}}.\end{equation}
This is because of the formula:
\begin{equation}\frac{d^m}{dx^m}\left(e^{-ax^2}\right)=a^{m/2}(-1)^me^{-ax^2}H_m(\sqrt{a}x).\end{equation}
Now we put
\begin{equation}2^{-2m}\frac{\sqrt{\pi}}{2}\sum_{n\ge1}(-1)^{n-1}\frac{1}{\sqrt{n}}\frac{e^{-b^2/n}H_m(\sqrt{\frac{1}{n}}b)}{n^{m/2}}=\sum_{n\ge1}h_{n,m}(b).\end{equation}
Applying the Poisson summation formula [1] for Fourier cosine transforms to the function
$$f(x)=\frac{x^{2m}}{e^{x^2}+1},$$ we get
$$\sum_{n\ge1}\frac{n^{2m}}{1+e^{n^2}}=\int_{0}^{\infty}\frac{x^{2m}}{1+e^{x^2}}dx+2\sum_{n\ge1}\sum_{k\ge1}h_{k,m}(\pi n)$$
$$=\int_{0}^{\infty}\frac{x^{m-\frac{1}{2}}}{1+e^{x}}dx+2\sum_{n\ge1}\sum_{k\ge1}h_{k,m}(\pi n)$$
$$=\sum_{n\ge1}(-1)^{n-1}\int_{0}^{\infty}x^{m-\frac{1}{2}}e^{-nx}dx+2\sum_{n\ge1}\sum_{k\ge1}h_{k,m}(\pi n)$$
$$=\Gamma(m+\frac{1}{2})(1-2^{\frac{1}{2}-m})\zeta(m+\frac{1}{2})+2\sum_{n\ge1}\sum_{k\ge1}h_{k,m}(\pi n).$$
\end{proof}
\section{Concluding Remarks}

We first remark that the summation formula in our Theorem 3.1 may be easily generalized, and the proof may be easily modified to give formulae for a variety of Dirchlet series [1]. We have also restricted ourselves to integrals related to (1.3) in this section, and it would be easy to include (1.4) as well. \par If we set $b=0$ (1.6), the integral reduces to (call it $F(a)$) \begin{equation}F(a)=\frac{(-1)^n}{2}\sqrt{\pi}(x)^{-(n+1/2)}e^{-a^2/4x}H_{2n}(a/\sqrt{4x}).\end{equation}
The integral (1.6) may now be viewed as $[F(a+b)+F(a-b)]/2.$ By comparing this with the right side of (1.6), we have the interesting identity
\begin{equation} H_{2n}(\frac{a+b}{\sqrt{4x}})+e^{ab/x}H_{2n}(\frac{a-b}{\sqrt{4x}})\end{equation}
$$=\left(\frac{b^2}{2x}\right)^n\sum_{k\ge0}^{2n}\binom{2n}{k}H_{k}(\frac{a}{\sqrt{4x}})\left(\frac{\sqrt{x}}{b}\right)^k[(-1)^ke^{ab/x}+1]. $$
\\*
A particularly interesting open problem is the following: Find an interpretation of Corollary 1.2.1 in terms of the wave packet.

{\bf Acknowledgement.} We thank M.L. Glasser for providing a copy of his work [4], and also noting the hidden gem (4.2), and its proof.

1390 Bumps River Rd. \\*
Centerville, MA
02632 \\*
USA \\*
E-mail: alexpatk@hotmail.com
\end{document}